\documentclass[12pt,letterpaper,english]{amsart}
\usepackage[T1]{fontenc}
\usepackage[latin9]{inputenc}
\usepackage{amsmath}
\usepackage{amssymb}
\usepackage{amsthm}
\usepackage{graphicx}
\usepackage{cancel}
\usepackage{enumitem}
\usepackage{setspace}
\usepackage{amsfonts}
\usepackage{mathrsfs}
\usepackage{amstext}
\usepackage[all]{xy}
\usepackage{array}
\usepackage{lmodern}
\usepackage{tikz}

\makeatletter

\numberwithin{equation}{section}
\numberwithin{figure}{section}
\theoremstyle{plain}
\newtheorem{thm}{\protect\theoremname}[section]
  \theoremstyle{plain}
  \newtheorem{lem}[thm]{\protect\lemmaname}
  \theoremstyle{remark}
  \newtheorem{rem}[thm]{\protect\remarkname}
  \theoremstyle{definition}
  
  \theoremstyle{plain}
  \newtheorem{cor}[thm]{\protect\corollaryname}
  \theoremstyle{plain}
  \newtheorem{prop}[thm]{Proposition}
  \theoremstyle{definition}
  \newtheorem{example}[thm]{\protect\examplename}

\def\QQ{\mathbb{Q}}
\def\RR{\mathbb{R}}
\def\CC{\mathbb{C}}
\def\ZZ{\mathbb{Z}}
\def\PP{\mathbb{P}}
\def\b{\bullet}
\def\ay{\mathbf{i}}
\def\B{\mathcal{B}}
\def\e{\epsilon}
\def\ee{\varepsilon}
\def\ue{\underline{\epsilon}}
\def\l{\lambda}
\def\ul{\underline{\lambda}}

\newcommand*{\@old@slash}{}\let\@old@slash\slash
\def\slash{\relax\ifmmode\delimiter"502F30E\mathopen{}\else\@old@slash\fi}

\newcommand{\noun}[1]{\textsc{#1}}

\makeatother

\usepackage{babel}
  \providecommand{\corollaryname}{Corollary}
  \providecommand{\examplename}{Example}
  \providecommand{\remarkname}{Remark}
\providecommand{\theoremname}{Theorem}
\providecommand{\lemmaname}{Lemma}

\begin{document}

\title[Integral regulators]{Two applications of the integral regulator}
\author{Matt Kerr and Muxi Li}
\subjclass[2000]{14C25, 14C30, 19E15}
\begin{abstract}
We review Li's refinement of the KLM regulator map, and use it to detect torsion phenomena in higher Chow groups.
\end{abstract}

\maketitle

\section{Introduction}

The KLM formula is a morphism of complexes inducing the Bloch-Beilinson regulator map with rational coefficients, developed by the first author together with J. Lewis and S. M\"uller-Stach \cite{Ke1,KLM,KL} (see $\S$3).  The second author's refinement now enables the direct computation of the \emph{integral} regulator on the level of higher Chow complexes \cite{Li}. In this note, we shall briefly review that construction ($\S$4) and show how it may be used to find explicit torsion generators in higher Chow groups of number fields ($\S$5).  We also apply the formula to \emph{integrally} calculate a branch of the higher normal function arising from the mirror of local $\PP^2$ ($\S$6).

\subsection*{Acknowledgments}

We thank the National Science Foundation for support under the aegis of FRG Grant DMS-1361147, and C. Weibel for helpful correspondence.

\section{Higher Chow groups}

Invented by Spencer Bloch \cite{Bl,Bl2} in the mid-1980s to geometrize Quillen's higher algebraic $K$-theory, these generalize the usual Chow groups of cycles modulo rational equivalence (the $n=0$ case). In particular, for $X$ smooth quasi-projective over an infinite field $k$, they satisfy $${CH}^p(X,n)\otimes \QQ \cong \mathrm{Gr}^p_{\gamma}K_n^{\text{alg}}(X)\otimes \QQ.$$ For such $X$, Voevodsky \cite{Vo} proved they were \emph{integrally} isomorphic to his motivic cohomology groups: $${CH}^p(X,n)\cong H^{2p-n}_{\mathcal{M}} (X,\ZZ(p)).$$
Beyond their role in arithmetic geometry (e.g. Beilinson's conjectures \cite{Be1}), they have recently shown up in several branches of physics (e.g. quantum field theory \cite{BKV1} and topological string theory \cite{7K}) and mirror symmetry \cite{DK2,BKV2}. We focus on the cubical presentation of ${CH}^p(X,n)$ as the $n^{\text{th}}$ homology of a complex of \emph{higher Chow precycles} \cite{Le}
$$\cdots \to Z^p(X,n+1)\overset{\partial}{\to} Z^p(X,n) \overset{\partial}{\to} Z^p(X,n-1)\to \cdots$$
or its (integrally quasi-isomorphic) subcomplex of \emph{normalized precycles} \cite{Bl4}
$$\cdots \to N^p(X,n+1)\overset{\partial}{\to} N^p(X,n) \overset{\partial}{\to} N^p(X,n-1)\to \cdots .$$

A \emph{higher Chow cycle} is an element of $\ker(\partial)$. Roughly speaking, these are relative codimension-$p$ cycles on
$$(X\times \mathbb{A}^n,X\times \cup \mathbb{A}^{n-1})$$
where the $\mathbb{A}^{n-1}$'s are inserted into $\mathbb{A}^n$ as a ``cubical'' configuration of hyperplanes. More precisely, writing
$$\square^n :=(\PP^1\setminus \{1\})^n \supset \partial \square^n := \bigcup_i \{z_i=0\text{ or }\infty\}$$
we set\footnote{Normalized precycles may be represented (in $Z^p(X,n)$) by $Z$ satisfying $Z\cdot\{z_i =0\}=0$ ($\forall i$) and $Z\cdot \{z_i=\infty\}=0$ ($i<n$) simply by adding an element of $d^p(X,n)$.}
\begin{flalign*}
c^p(X,n) &:= \left\{\text{cycles meeting faces of }X\times\partial\square^n\text{ properly}\right\}\\
d^p(X,n) &:= \left\{\text{cycles ``constant'' in some }z_i\right\}\\
Z^p(X,n) &:= c^p(X,n)/d^p(X,n)\\
N^p(X,n) &:= \left\{Z\mid Z\cdot\{z_i=0\}=Z\cdot\{z_i=\infty\}=0\, (\forall i<n)\right\}
\end{flalign*}
and for $Z\in Z^p(X,n)$ or $N^p(X,n)$,
$$\partial Z := \sum_{i=1}^n (-1)^i \left( Z\cdot \{z_i=\infty\}-Z\cdot \{z_i=0\}\right).$$
If $X=\text{Spec}(k)$, write $Z^p(k,n)$ etc. for short.

\begin{example} \label{ex1}
Parametrize a cycle in $N^2(\QQ(\zeta_{\ell}),3)$ by $t\in \PP^1$:
$$Z^2_{\ell} :=\left( 1-\frac{\zeta_{\ell}}{t},1-t,t^{-{\ell}}\right) . $$
Intersections with facets $\{z_i=0,\infty\}$ are given by $t=0,1,\zeta_{\ell},\infty$. But all these intersections have some $z_j=1$, so are trivial (as $1\notin\square$).
We also record the cycle 
$$\mathscr{Z}^2_5 :=Z^2_1 +\left( 1-\frac{\zeta_5}{t},1-t,t^{-5}\right)+\left(1-\frac{\overline{\zeta_5}}{t},1-t,t^{-5}\right)$$
in $N^2(\QQ(\sqrt{5}),3)$ for later reference.
\end{example}

\section{Abel-Jacobi maps}

These simultaneously generalize two classical invariants:

\begin{enumerate}
\item Griffiths's AJ map \cite{Gr}
$${CH}^p(X,0)\to H_{\mathscr{D}}^{2p}(X,\ZZ(p))$$ for $X$ smooth projective over $\CC$; and
\item ${}$[A $\ZZ$-lift of] Borel's regulator map \cite{Bo2,Bu} $${CH}^p(k,2p-1)\to \CC/\ZZ(p)$$ for $k\subset\CC$ a number field.
\end{enumerate}

\noindent Defined abstractly by Bloch \cite{Bl3}, they map higher Chow groups to Deligne cohomology:\footnote{or (better) to absolute Hodge cohomology \cite[$\S$2]{KL} in the smooth quasiprojective case.}
$${CH}^p(X,n)\overset{{AJ}^{p,n}}{\to}H_{\mathscr{D}}^{2p-n}(X,\ZZ(p)).$$
Kerr, Lewis, and M\"uller-Stach \cite{KLM} constructed a morphism of complexes
\begin{flalign*}
\widetilde{AJ}^{p,-\bullet}_{\text{KLM}}:\, Z^p_{\RR}(X,-\bullet) & \;\;\longrightarrow\;\;C_{\mathscr{D}}^{2p+\bullet}(X,\ZZ(p)):=
\\
& C_{\text{sing}}^{2p+\bullet}(X;\ZZ(p))\oplus F^p D^{2p+\bullet}(X)\oplus D^{2p-1+\bullet}(X)
\end{flalign*}
with differential $D(\alpha,\beta,\gamma)=(-\partial\alpha,-d\beta,d\gamma-\beta+\alpha)$ on the right.
For $Z\in Z^p_{\RR}(X,n)$ with projections $\pi_1$ (to $\square^n$) and $\pi_2$ (to $X$), 
they define
\begin{flalign*}
\widetilde{AJ}^{p,n}_{\text{KLM}}(Z) &:= (2\pi \ay)^{p-n} \left( (2\pi\ay)^n T_Z,\Omega_Z,R_Z\right)\\
&:= (2\pi\ay)^{p-n} (\pi_2)_*(\pi_1)^*\left( (2\pi\ay)^n T_n,\Omega_n,R_n\right)
\end{flalign*}
where
$T_n:=\bigcap_{i=1}^n T_{z_i} = \RR_{<0}^{\times n}$, \;\;$\Omega_n :=\tfrac{dz_1}{z_1}\wedge \cdots \wedge \tfrac{dz_n}{z_n}$, \;and
$$R_n := \log(z_1)\tfrac{dz_2}{z_2}\wedge \cdots \wedge \tfrac{dz_n}{z_n} - (-1)^n (2\pi\ay)R_{n-1}\cdot \delta_{T_{z_1}} .$$
Here $\log(z)$ has a branch cut along $T_z = \{z\in\RR_{<0}\}$, and $R_1 = \log(z)$. (Note that $\widetilde{AJ}^{p,n}_{\text{KLM}}$ vanishes identically on $d^p(X,n)$.)

The subcomplex $Z^p_{\RR}(X,-\bullet)\subset Z^p(X,-\bullet)$ consists of cycles $Z$ for which $Z^{\text{an}}$ properly intersects the various combinations of $\{T_{z_i}\}$ and $\{z_j=0,\infty\}$. We call such precycles $\RR$-\emph{proper}.
Kerr and Lewis \cite{KL} proved the inclusion is a \emph{rational} quasi-isomorphism, by appealing to Kleiman transversality in $K$-theory. Unfortunately, the claimed integral moving lemma in \cite{KLM} (which would have made this quasi-isomorphism integral) was incorrect, and \cite{KL} was only written after a prolonged effort to repair the integral version.

Now suppose we have a cycle $Z\in \ker(\partial)\subset Z^p_{\RR}(X,n)$ with $$[\widetilde{AJ}^{p,n}_{\text{KLM}}(Z)]\in H^{2p-n}_{\mathscr{D}}(X,\ZZ(p))$$ torsion of order $M$.  This implies $[Z]\in H_n\{Z^p_{\RR}(X,\bullet)\}$ is at least of this order. But for $[Z]\in H_n\{Z^p(X,\bullet)\}={CH}^p(X,n)$, it means no such thing: there could be a $W\in Z^p(X,n+1)\setminus Z^p_{\RR}(X,n+1)$ with $\partial W=Z$. So the KLM map only induces a homomorphism
$${AJ}^{p,n}_{\QQ}:\,{CH}^p (X,n)\to H_{\mathscr{D}}^{2p-n}(X,\QQ(p))$$
consistent with Bloch's ${AJ}^{p,n}$. This is frustrating, as the KLM formulas are well-adapted to detecting torsion!

For $X=\text{Spec}(k)$ and $(p,n)=(2,3)$, consider the portion
\[\xymatrix@C=1em{
\cdots \ar [r] & Z^2_{\RR}(k,4) \ar [r] \ar [d]^{(2\pi\ay)^2 W\cdot T_{4}} &  Z^p_{\RR}(k,3) \ar [r]\ar[d]^{\frac{1}{2\pi\ay}\int_Z R_{2p-1}} & Z^p_{\RR}(k,2) \ar [r] \ar[d]^0 & \cdots \\
\cdots \ar [r] & \ZZ(2)  \ar @{^(->} [r] & \CC \ar [r] & 0 \ar [r] & \cdots 
}\]
of the KLM map of complexes. We want to use the middle map to detect torsion. Denote its image on a cycle $Z$ by $\mathscr{R}(Z)\in\CC/\ZZ(2)$.

\begin{example}[Petras \cite{Pe2}] \label{ex2}

We calculate $\mathscr{R}(Z_{\ell}^2)=$
\begin{flalign*}
\frac{1}{2\pi\ay}\int_{Z^2_{\ell}}R_3 &=\frac{1}{2\pi\ay} \int_{Z_{\ell}^2} \begin{pmatrix}\log(z_1) {dz_2}/{z_2}\wedge {dz_3}/{z_3} + \\ (2\pi\ay) \log(z_2){dz_3}/{z_3} \cdot \delta_{T_{z_1}} \\ +(2\pi\ay)^2 \log(z_3) \delta_{T_{z_1} \cap T_{z_2}} \end{pmatrix} \\  
&=\int_{Z_{\ell}\cap T_{z_1}} \log(z_2)\frac{dz_3}{z_3} \; = \; -\int_{T_{1-\frac{\zeta_{\ell}}{t}}}\log(1-t)\frac{dt}{t}   \\ 
&=-\int_0^{\zeta_{\ell}} \log(1-t) \frac{dt}{t} \;=\; \text{Li}_2(\zeta_{\ell}).
\end{flalign*}
For $\ell=1$, this is $\frac{\pi^2}{6}\in\CC/\ZZ(2)$, which is $24$-torsion, while (for the second cycle of Example \ref{ex1})  $\mathscr{R}(\mathscr{Z}^2_5)=\text{Li}_2(1)+\text{Li}_2(\zeta_5)+\text{Li}_2(\overline{\zeta_5})=\frac{7\pi^2}{30}$ is $120$-torsion.  To \emph{deduce} that these orders of torsion exist in ${CH}^2(\QQ,3)$ resp. ${CH}^2(\QQ(\sqrt{5}),3)$, we need an improvement in technology.
\end{example}

\section{The integral regulator}

A few basic strategies come to mind:\vspace{0.1cm}

\noindent {\bf (1)} proving an integral moving lemma ($Z^p_{\RR}(X,\bullet)\overset{\simeq}{\to}Z^p(X,\bullet)$)

\noindent and

\noindent {\bf (2)} extending KLM to a map of complexes on $Z^p(X,\bullet)$

\vspace{0.1cm}

\noindent are probably too naive;

\vspace{0.1cm}

\noindent{\bf (3)} extending KLM to an infinite family of homotopic maps on nested subcomplexes with union $Z^p(X,\bullet)$

\vspace{0.1cm}

\noindent seemed promising; but what ultimately worked was

\vspace{0.1cm}

\noindent{\bf (4)} extending KLM to an infinite family of homotopic maps on nested subcomplexes with union $N^p(X,\bullet).$

The heuristic idea of {\bf (3)} was to perturb the branch cuts $T_{z_i}=\{ z_i \in \RR_{<0}\}$ in $\log(z_i)$ to $T_{z_i}^{\e}=\{ z_i /e^{\ay\e}\in \RR_{<0}\}$ and take a limit as $\e\to 0$, an approach that had been successfully applied in \cite[$\S$9]{Ke}. Unfortunately, there are cycles in $Z^2(\CC,3)$ whose intersection with $T_{z_1}^{\e}\cap T_{z_2}^{\e}\cap T_{z_3}^{\e}$ is improper for every real $\e$ near $0$ \cite[$\S$3]{Li}. So we need to deform the branches by distinct $\{\e_i\}$; but then we cannot expect a morphism of complexes (or ``limit'' thereof) on $Z^p(X,\bullet)$.  This forces us into strategy {\bf (4)}, and working with normalized subcomplexes.

Let $\B_{\ee}$ denote the set of infinite sequences $\{\e_i\}_{i>0}$, with
$$0<\e_1<\ee,\;\;\;0<\e_2<e^{-1/\e_1},\;\;\;0<\e_3<e^{-1/\e_2},\; \text{etc.},$$
so that when $\ee\to 0$ its projection to any $(S^1 )^n$ eventually avoids any given analytic subvariety. Let $N^p_{\ee}(X,\b)\subset N^p(X,\b)$ denote the (nested) subcomplexes of cycles $Z$ with $Z^{\text{an}}$ properly intersecting (for each $\ue\in \B_{\ee}$) certain\footnote{namely, $\left( \cap_{i\in I}\{ T^{\epsilon_i}_{z_i}\}\right) \cap \left( \cap_{j\in J}\{ z_j =0,\infty\}\right)$ where $I$ and $K\setminus I$ are consecutive (no gaps) in $K=\{1,\ldots,n\}\setminus J$ (e.g., $T_{z_1}^{\epsilon_1}\cap \{z_2 =0\} \cap T^{\epsilon_3}_{z_3}\cap\{z_7=\infty\}$).} combinations of $\{T_{z_i}^{\e_i}\}$ and $\{z_j =0,\infty\}$.

\begin{lem}[\cite{Li}, Thms. 4.2 and 7.2] \label{lem1}
We have $$\bigcup_{\ee>0} N^p_{\ee}(X,n)=N^p(X,n)\;\;\;\;(\forall n)$$
and
$$\lim_{\ee\to 0}H_n(N^p_{\ee}(X,\b))\cong H_n(N^p(X,\b))\cong {CH}^p(X,n).$$
\end{lem}

\noindent For any $\ue\in \B_{\ee}$, replacing $T_{z_i}$ by $T_{z_i}^{\e_i}$ everywhere in the KLM formula yields a morphism of complexes

$$\widetilde{AJ}^{p,-\b}_{\ee,\ue}:\;N^p_{\ee}(X,-\b)\longrightarrow C_{\mathscr{D}}^{2p+\b}(X,\ZZ(p)).$$

\begin{lem}[\cite{Li}, Thm. 6.1] \label{lem2}
Given $\ue,\ue'\in\B_{\ee}$, $\widetilde{AJ}^p_{\ee,\ue}$ and $\widetilde{AJ}^p_{\ee,\ue'}$ are \textup{(}$\ZZ$-\textup{)}homotopic.
\end{lem}

\begin{proof}[Sketch] 
Truncating at some $N$, we may view
$$\mathcal{R}^{\ue}_{\square} =\left\{ \mathcal{R}_n^{\hat{\ue}}:=\left( (2\pi\ay)^n T_n^{\hat{\ue}},\Omega_n,R_n^{\hat{\ue}}\right)\right\}_{n,\hat{\ue}}$$
($0\leq n\leq N$; $\{\hat{\e}_1,\ldots,\hat{\e}_n\}\subset \{\e_1,\ldots,\e_N\}$ subsequence) as a $0$-cocycle in the double complex
$$\left( E^{a,b}=C_{\mathscr{D}}^{2a+b}\left( (\PP^1)^a\right)^{\oplus {N \choose a} 2^{N-a}},\;\delta_{\text{Gysin}},\;D_{\mathscr{D}}\right) .$$
Construct a $(-1)$-cochain $\mathcal{S}^{\ue,\ue'}_{\square}$ with $\mathbb{D}\mathcal{S}^{\ue,\ue'}_{\square} = \mathcal{R}^{\ue} - \mathcal{R}^{\ue'}$, and with respect to whose wavefront set the precycles in $N^p_{\ee}(X,\b)$ remain proper.
\end{proof}

\noindent We therefore have well-defined, compatible maps
$${AJ}^{p,n}_{\ee}:\;H_n(N^p_{\ee}(X,\b))\to H_{\mathscr{D}}^{2p-n}(X,\ZZ(p))$$
for each $\ee>0$, essentially given by $\lim_{\ue\to \underline{0}}\widetilde{AJ}^p_{\ue}$ (where the limit is taken so that $\ee>\e_1\gg\e_2\gg\e_3\gg \cdots >0$), and recovering $\widetilde{AJ}^p_{\text{KLM}}$ on $N^p_{\RR}(X,\b) := Z^p_{\RR}(X,\bullet)\cap N^p(X,\bullet)$.  

More precisely:
\begin{thm}[\cite{Li}, $\S$7]
The $\{ \widetilde{AJ}^{p,-\b}_{\ee,\ue}\}$ induce a homomorphism
$${AJ}^{p,n}_{\ZZ}:\;{CH}^p(X,n)\cong \lim_{\ee\to 0}H_n(N^p_{\ee}(X,\b))\to H_{\mathscr{D}}^{2p-n}(X,\ZZ(p))$$
factoring ${AJ}^{p,n}_{\QQ}$.
\end{thm}

\noindent This theorem has the
\begin{cor}
If a class $\xi\in  {CH}^p(X,n)$ is represented by $$Z\in \ker(\partial)\subset N^p_{\RR}(X,n)\, \left( \subset \bigcap_{\varepsilon>0} N_{\varepsilon}^p(X,n) \right) ,$$
then
$${AJ}^{p,n}_{\ZZ}(\xi )=\lim_{\ue\to \underline{0}}\widetilde{AJ}^{p,n}_{\ue}(Z)=\widetilde{AJ}_{\text{KLM}}^{p,n}(Z).$$
\end{cor}

\noindent So the KLM formula holds verbatim on normalized, $\RR$-proper representatives, validating the deductions at the end of Example \ref{ex2}. It is this statement that we (primarily) use in the applications that follow.

\section{Torsion generators}

Let $\mu_{\infty}=\bigcup_{m\in \mathbb{N}} \mu_m \subset \CC^*$ denote the roots of unity, and $w_r(k):=\left|(\mu_{\infty}^{\otimes r})^{\mathrm{Gal}(\bar{\QQ}/k)}\right|$ for any number field $k\subset \CC$. By the universal coefficient sequence for motivic cohomology
$$H^0_{\mathcal{M}}(k,\ZZ(r))\to H^0_{\mathcal{M}}(k,\ZZ/m\ZZ(r))\to H^1_{\mathcal{M}}(k,\ZZ(r))\overset{\cdot m}{\to}H^1_{\mathcal{M}}(k,\ZZ(r))$$
and vanishing of $H^0_{\mathcal{M}}(k,\ZZ(r))$,\footnote{See \cite[Ex. VI.4.6]{We}: since $H^{-1}_{\mathcal{M}}(k,\ZZ/m\ZZ(r))\cong H^{-1}_{\text{\'et}}(k,\mu_m^{\otimes r})=\{0\}$, $\cdot m$ is injective on $H^0_{\mathcal{M}}(k,\ZZ(r))$ (universal coefficient sequence), which is thus torsion-free; it has rank $0$ since $H^0_{\mathcal{M}}(k,\QQ(r))\cong \mathrm{Gr}^r_{\gamma}K_{2r}(k)\otimes \QQ=\{0\}$ by Borel's theorem \cite{Bo1}.} we have
$$CH^r(k,2r-1)[m]\cong H^1_{\mathcal{M}}(k,\ZZ(r))[m]\cong H^0_{\mathcal{M}}(k,\ZZ/m\ZZ(r)).$$
Since the norm residue map
$$H^0_{\mathcal{M}}(k,\ZZ/m\ZZ(r))\to H^0_{\text{\'et}}(k,\mu_m^{\otimes r})\cong (\mu_m^{\otimes r})^{\mathrm{Gal}(\bar{\QQ}/k)}$$
is an isomorphism by a celebrated theorem of Rost-Voevodsky (cf. \cite{HW}), we conclude that $CH^r(k,2r-1)[m]\cong \ZZ/(m,w_r(k))\ZZ$ hence
$$CH^r(k,2r-1)_{\text{tors}}\cong \ZZ/w_r(k)\ZZ.$$

\begin{example} \label{ex5.1}
If $k=\QQ$, one has\footnote{Bernoulli numbers: $|B_{2n}|=\frac{1}{6},\frac{1}{30},\frac{1}{42},\ldots$ for $n=1,2,3,\ldots$.} $w_{2n}(k)=$ denominator of $\frac{|B_{2n}|}{4n}$ (written in lowest terms) and $w_{2n+1}=2$ for $n\geq 1$; so
$$CH^2(\QQ,3)\cong \ZZ/24\ZZ,\;\;CH^3(\QQ,5)_{\text{tors}}\cong \ZZ/2\ZZ,\;\;CH^4(\QQ,7)\cong \ZZ/240\ZZ.$$
For real quadratic fields $k=\QQ(\sqrt{d})$, the situation is more complicated (cf. \cite[$\S$VI.2]{We}); one computes for instance
\begin{center}
\begin{tabular}{m{5em}|m{1cm}|m{1cm}|m{1cm}|m{1cm}}
$d$ & $2$ & $3$ & $5$ & $7$ \\
\hline
$w_2(\QQ(\sqrt{d}))$ & $48$ & $24$ & $120$ & $24$ \\
\hline
$w_4(\QQ(\sqrt{d}))$ & $480$ & $240$ & $240$ & $240$
\end{tabular}
\end{center}
so that only $CH^2(\QQ(\sqrt{2}),3)$ and $CH^2(\QQ(\sqrt{5}),3)$ [resp. $CH^4(\QQ(\sqrt{2}),7)$] are different from the $k=\QQ$ case.  Finally, for cyclotomic $k=\QQ(\zeta_p)$ ($p\geq 5$ prime) one can show that $w_3(\QQ(\zeta_p))=2p$, while $w_3 (\QQ(\zeta_3))=18$.
\end{example}

For computing torsion orders of images under
\begin{flalign*}
AJ^{r,2r-1}_{\ZZ}:\,H_{2r-1}(N^r_{\RR}(k,&\bullet)) \to \CC/(2\pi\ay)^r\ZZ\\
&Z\longmapsto \frac{1}{(2\pi\ay)^{r-1}}\int_Z R_{2r-1} =: \mathscr{R}(Z)
\end{flalign*}
we use the following basic calculation:

\begin{prop} \label{prop5}
Suppose that for a given $r\in \mathbb{N}$ there exists a collection of closed precycles $Z^r_{\ell,a}\in N^r_{\RR}(\QQ(\zeta_{\ell}),2r-1)$ with \begin{equation}\label{eq5.0}\mathscr{R}(Z^r_{\ell,a})=(r-1)!\ell^{r-1}\mathrm{Li}_r(\zeta_{\ell}^a).\end{equation} Then for $Z:=\sum_{a=0}^{\ell-1} f(a) Z^r_{\ell,a}$ with $f(-a)=(-1)^r f(a)$, $AJ(Z)$ is torsion of order given by the denominator of \begin{equation}\label{eq5.1} \tau(Z):= \left| (2\pi\ay)^{-r} \mathscr{R}(Z)\right| = \pm\frac{\ell^{r-1}}{2r} \sum_{a=0}^{\ell-1} f(a) B_r(\tfrac{a}{\ell}).\end{equation}
\end{prop}

\begin{proof}
By \cite[Thm. 3.9]{css},
\begin{flalign*}
\sum_{a=0}^{\ell-1} f(a) \mathrm{Li}_r(\zeta_{\ell}^a) & = \frac{1}{2} \sum_{a=0}^{\ell-1} f(a)\sum_{k\in \ZZ\setminus \{0\} } \frac{\zeta_{\ell}^{ka}}{k^r} \\
&=: \frac{(-1)^r}{2}\sum_{k\in \ZZ\setminus \{0\} } \frac{\hat{f}(k)}{k^r} =: \frac{(-1)^r}{2}\tilde{L}(\hat{f},r)\\
&=\frac{(2\pi\ay)^r}{2\cdot r!}\sum_{a=0}^{\ell-1} f(a) B_r(\tfrac{a}{\ell}),
\end{flalign*}
where $B_r(\cdot)$ are the Bernoulli polynomials.\footnote{$B_2(x)=x^2 -x+\frac{1}{6}$, $B_3(x)=x^3 - \frac{3}{2}x^2 +\frac{1}{2}x$, $B_4(x)=x^4 - 2x^3 + x^2 -\frac{1}{30}$, etc.}
\end{proof}

In practice, the $\{Z^r_{\ell,a}\}$ will be obtained from a single cycle $Z_{\ell}^r = Z^r_{\ell,1}$ by Galois conjugation. For $r=2$, we already have this from Examples \ref{ex1} and \ref{ex2}.

Now $CH^r(k,2r-1)=CH^r(k,2r-1)_{\text{tors}}$ $\iff$ $r$ is even and $k$ is totally real. In particular, assuming Prop. \ref{prop5}'s hypothesis, we obtain generators of $CH^{2n}(k,4n-1)$ as follows:

\begin{itemize}
\item $k=\QQ$, any $n$: $\;\;\tau(Z^{2n}_1)=\frac{|B_{2n}|}{4n} \, (=\tfrac{1}{24},\tfrac{1}{240},\cdots)$;
\item $k=\QQ(\sqrt{2})$, $n=1,2$: $\;\;\tau(Z^{2n}_{8,1}+Z^{2n}_{8,7})=\tfrac{11}{48},\tfrac{1313}{480}$;
\item $k=\QQ(\sqrt{5})$, $n=1$: $\;\;\tau(Z_1^2 + Z_{5,1}^2 + Z_{5,4}^2 )=\tfrac{7}{120}$.
\end{itemize}
For $CH^3(k,5)_{\text{tors}}$ one computes for example
\begin{itemize}
\item $k=\QQ(\zeta_3)$: $\;\;\tau(Z^3_{3,1}-Z^3_{3,2}) = \tfrac{1}{9}\;$ and
\item $k=\QQ(\zeta_5)$: $\;\;\tau(Z_{5,1}^3-Z_{5,4}^3)=\tfrac{2}{5}$,
\end{itemize}
which miss only the 2-torsion element from $CH^3(\QQ,5)\hookrightarrow CH^3(k,5)$. (So far we have no $N^3_{\RR}(\QQ,5)$ representative for this element.)

It remains to construct the cycles of the Proposition for $r=3,4$. From \cite[$\S$4.2]{KY}, for $r=3$ we have\footnote{The components are parametrized by $(t_1,t_2)$ and $(t,u)$ respectively.}
\begin{multline*}
Z_{\ell}^3 := -2 \left( \frac{t_1}{t_1 -1},\frac{t_2}{t_2 -1},1-\zeta_{\ell}t_1 t_2, t_1^{\ell},t_2 ^{\ell}\right) \\ -\left( \frac{t}{t-1},\frac{1}{1-\zeta_{\ell}t},\frac{(u-t^{\ell})(u-t^{-\ell})}{(u-1)^2},t^{\ell}u,\frac{u}{t^{\ell}}\right)
\end{multline*}
which is normalized since all ``boundaries'' occur in the third coordinate. We have $\mathscr{R}(Z_{\ell}^3)=2\mathrm{Li}_3(\zeta_{\ell})$ by [op. cit., Thm. 3.6], with only the first term contributing. (This gives in particular $\mathscr{R}(Z_1^3)=2\mathrm{Li}_3(1)=2\zeta(3)$.)

For $r=4$, the first construction in \cite[$\S$4.3]{KY} would be in $N^4_{\RR}(\QQ(\zeta_{\ell}),7)$, but there is an error in the computation of the boundary of the last component $\mathscr{W}_2$: in fact, it is degenerate,\footnote{i.e. belongs to $d^4(\QQ(\zeta_{\ell}),7)$, as can be seen by substituting $v=uw$.} and the cycle $\tilde{\mathscr{Z}}$ is therefore not closed. A correct application of the strategy in [op. cit., $\S$3.1] yields $Z_{\ell}^4 :=$
\begin{flalign*}
& 6\left( \frac{t_1}{t_1-1},\frac{t_2}{t_2-1},\frac{t_3}{t_3-1},1-\zeta_{\ell}t_1 t_2 t_3,t_1^{\ell},t_2^{\ell},t_3^{\ell} \right) \\
+& \left(\frac{t_1}{t_1-1},\frac{t_2}{t_2-1},\frac{1}{1-\zeta_{\ell}t_1 t_2}, \frac{(u-t_1^{\ell})(u-t_2^{\ell})(u-t_1^{-\ell}t_2^{-\ell})}{(u-1)^3},\frac{t_1^{\ell}}{u},\frac{t_2^{\ell}}{u},\frac{1}{ut_1^{\ell}t_2^{\ell}} \right) \\
+& \left(\frac{t_1}{t_1-1},\frac{t_2}{t_2-1},\frac{1}{1-\zeta_{\ell}t_1 t_2}, \frac{(u-t_1^{\ell})(u-t_2^{\ell})}{(u-t_1^{\ell}t_2^{\ell})(u-1)},\frac{t_1^{\ell}}{u},\frac{t_2^{\ell}}{u},\frac{u}{t_1^{\ell}t_2^{\ell}} \right) \\
+& \left(\frac{t_1}{t_1-1},\frac{t_2}{t_2-1},\frac{1}{1-\zeta_{\ell}t_1 t_2}, \frac{(u-t_1^{\ell})(u-t_1^{-\ell}t_2^{-\ell})}{(u-t_2^{-\ell})(u-1)},\frac{t_1^{\ell}}{u},t_2^{\ell}u,\frac{1}{ut_1^{\ell}t_2^{\ell}} \right) \\
+& \left(\frac{t_1}{t_1-1},\frac{t_2}{t_2-1},\frac{1}{1-\zeta_{\ell}t_1 t_2}, \frac{(u-t_2^{\ell})(u-t_1^{-\ell}t_2^{-\ell})}{(u-t_1^{-\ell})(u-1)},t_1^{\ell}u,\frac{t_2^{\ell}}{u},\frac{1}{t_1^{\ell}t_2^{\ell}u} \right) \\
+&\left(\frac{(v-u)(v-u^{-1})}{(v-1)^2}, \frac{t}{t-1},\frac{1}{1-\zeta t},\frac{(u-t^{\ell})(u-t^{-\ell})}{(u-1)^2},v,\frac{t^{\ell}}{u},\frac{1}{u t^{\ell}} \right)\\
+&\left( \frac{t}{t-1}, \frac{(v-u)(v-u^{-1})}{(v-1)^2},\frac{1}{1-\zeta t},\frac{(u-t^{\ell})(u-t^{-\ell})}{(u-1)^2},\frac{t^{\ell}}{u},v,\frac{1}{u t^{\ell}} \right)\\
+&\left( \frac{t}{t-1},\frac{1}{1-\zeta t},\frac{(v-u)(v-u^{-1})}{(v-1)^2},\frac{(u-t^{\ell})(u-t^{-\ell})}{(u-1)^2},\frac{t^{\ell}}{u},\frac{1}{u t^{\ell}},v \right) \;,
\end{flalign*}
which belongs to $\ker(\partial)\cap N^4_{\RR}(\QQ(\zeta_{\ell}),7)$. Only the first term contributes to $\mathscr{R}(Z^4_{\ell})=-6\ell^3 \int_{[0,1]^3} \log(1-\zeta_{\ell}t_1 t_2 t_3)\frac{dt_1}{t_1}\wedge \frac{dt_2}{t_2}\wedge \frac{dt_3}{t_3} = 6\ell^3 \mathrm{Li}_4(\zeta_{\ell})$, see [op. cit., $\S$3.2].

\begin{rem}
An example of a cycle for which the log-branch perturbations \emph{are} required for the integral regulator computation is $Z:=$ $$ Z_- - Z_+:=\left( \left(\frac{z-\ay}{z+\ay}\right)^{-2},\left(\frac{z-1}{z+1}\right)^{-2},z^{-2}\right) - \left( \left(\frac{z-\ay}{z+\ay}\right)^2,\left(\frac{z-1}{z+1}\right)^2,z^2\right)$$
(parametrized by $z\in \PP^1$) in $N^2(\QQ(\ay),3)$. Indeed, $T_{z^2}$ [resp. $T_{\left(\frac{z-1}{z+1}\right)^2}$, $T_{\left(\frac{z-\ay}{z+\ay}\right)^2}$] has support on $\ay\RR$ [resp. the unit circle $S^1$], so that the triple intersection (essentially $\ay\RR\cap S^1 \cap S^1$) is nonempty. Though this cycle is \emph{non}-torsion, we briefly describe the computation. After making the deformation, $T_{\left(\frac{z-\ay}{z+\ay}\right)^2}$ and $T_{\left(\frac{z-1}{z+1}\right)^2}$ intersect twice with opposite orientations, at points near $\ay$ and $-\ay$ with phase just greater than $\tfrac{\pi}{2}$ resp. $\tfrac{3\pi}{2}$. Since $\epsilon_3 \to 0$ much faster than $\epsilon_1$ and $\epsilon_2$, in the limit the $$2\pi\ay\int_{Z_+} \log^{\epsilon_3}(z^2)\,\delta_{T^{\epsilon_1}_{\left(\frac{z-\ay}{z+\ay}\right)^2}\cap T^{\epsilon_2}_{\left(\frac{z-1}{z+1}\right)^2}}$$ term of $\frac{1}{2\pi\ay}\int_{Z_+}R_3$ contributes $2\pi\ay(\pi\ay-\pi\ay)=0$. The remaining term yields 
\begin{equation}\label{eq5r} 2\int_{Z_+}\log^{\epsilon_2}\left(\left(\frac{1-z}{1+z}\right)^2\right)\frac{dz}{z}\,\delta_{T^{\epsilon_1}_{\left(\frac{z-\ay}{z+\ay}\right)^2}}\;, \end{equation} where $T^{\epsilon_1}_{\left(\frac{z-\ay}{z+\ay}\right)^2}$ consists of two paths from $-\ay$ to $\ay$, along which one checks that (in the limit) $\log^{\epsilon_2}\left(\left(\frac{1-z}{1+z}\right)^2\right) = 2\log^{\epsilon_2}(1-z)-2\log^{\epsilon_2}(1+z)$; and so \eqref{eq5r} becomes $8\int_{-\ay}^{\ay} \log(1-z)\frac{dz}{z}-8\int_{-\ay}^{\ay} \log(1+z)\frac{dz}{z}$.  Combining this with the portion from  $Z_-$, we obtain
$$\mathscr{R}(Z)=32\mathrm{Li}_2(\ay)-32\mathrm{Li}_2(-\ay)=64\ay L(\chi_4,2)\in\CC/\ZZ(2).$$
\end{rem}

\section{Local $\PP^2$ revisited}

For a reflexive polytope $\Delta \subset \RR^2$ with polar polytope $\Delta^{\circ}$, the mirror of $K_{\PP_{\Delta^{\circ}}}$ (``local $\PP_{\Delta^{\circ}}$'') can be identified with a family of $CH^2(\cdot,2)$-elements on a family of anticanonical (elliptic) curves in $\PP_{\Delta}$ \cite[$\S5$]{DK1}. As a second application of the integral regulator, we show how to apply it to compute the correct ``torsion term'' in the higher normal function associated to one of these families. That is, if $\mathcal{E}\overset{\pi}{\to}\PP^1_t$ is smooth away from $\Sigma = \{0\}\cup \Sigma^*$, with fibers $E_t = \pi^{-1}(t)$, and $\Xi\in CH^2(\mathcal{E}\setminus E_0,2)$ has fiberwise restrictions $\xi_t \in CH^2(E_t,2)$ ($t\notin \Sigma$), we shall compute 
\begin{flalign*}
\mathscr{R}_t := AJ^{2,2}(\xi_t)\in H_{\mathscr{D}}(E_t,\CC/\ZZ(2))&\cong H^1(E_t,\CC/\ZZ(2)) \\
&\cong \mathit{Hom}\left( H_1(E_t,\ZZ),\CC/(2\pi\ay)^2 \ZZ\right)
\end{flalign*}
in a neighborhood of $t=0$. Writing $\{\omega_t\}$ for a section of $\omega_{\mathcal{E}/\PP^1}$ vanishing at $\infty$, the constant term in (a branch of) the resulting truncated higher normal function $$\nu(t):=\tfrac{1}{2\pi\ay}\langle \omega_t,\mathscr{R}_t\rangle$$ will play a role in forthcoming work of the first author with C. Doran on quantum curves.

To begin in a somewhat more general scenario, let $\Delta \subset \RR^2$ be any convex polytope with integer vertices $\{ p_i = (a_i,b_i)\}_{i=1}^N$ and interior integer points $\{(v_j,w_j)\}_{j=1}^g$. Define a multiparameter family $$\rho:\mathcal{C}\to \CC^g\; , \;\; \;\;\; C_{\underline{\lambda}}:=\rho^{-1}(\underline{\lambda})$$ of (where smooth) genus $g$ curves by taking the Zariski closure of $\mathcal{C}^* :=$ $$\{ (x,y;\lambda_1,\ldots ,\lambda_g)\mid 0=\Phi_{\underline{\lambda}}(x,y):=\phi(x,y)-\Sigma_{j=1}^g \lambda_j x^{v_j} y^{w_j} \} \subset (\CC^*)^2 \times \CC^g$$ in $\PP_{\Delta}\times \CC^g$, where $\phi(x,y):=\sum_{(a,b)\in \partial\Delta \cap \ZZ^2} m_{a,b} x^a y^b$ is uniquely determined by requiring its edge polynomials to be powers of $(t+1)$.  (If the edges of $\Delta$ have no interior points, then $\phi(x,y)=\sum_{i=1}^N x^{a_i} y^{b_i}$.) The symbol $\{-x,-y\}$ represents a closed precycle $\Xi^*\in Z^2(\mathcal{C}^*,2)$ parametrized (in $\mathcal{C}^* \times \square^2$) by $(x,y,\ul,-x,-y)_{(x,y,\ul)\in \mathcal{C}^*}$.

\begin{lem}
The class of $\Xi^*$ in $CH^2(\mathcal{C}^*,2)$ is the restriction of a class $\Xi \in CH^2(\mathcal{C},2)$.
\end{lem}
\begin{proof}
We need only check that the Tame symbol of $\{-x,-y\}|_{C_{\underline{\lambda}}^*}\in K^M_2(\CC(C_{\underline{\lambda}}))$ is zero for general $\underline{\lambda}$. The symbol $\{-x,-y\}$ is invariant under unimodular change of toric coordinates,\footnote{that is, replacing $x,y$ by $x^a y^b, x^c y^d$ with $ad-bc=1$; the $a_i,b_i,v_i,w_i$ are changed accordingly.} so we may assume that (after shifting $\Delta$ by $(-a_m,-b_m)$ for some $m$) we have a picture
\begin{equation}\label{fig!}
\includegraphics[scale=0.5]{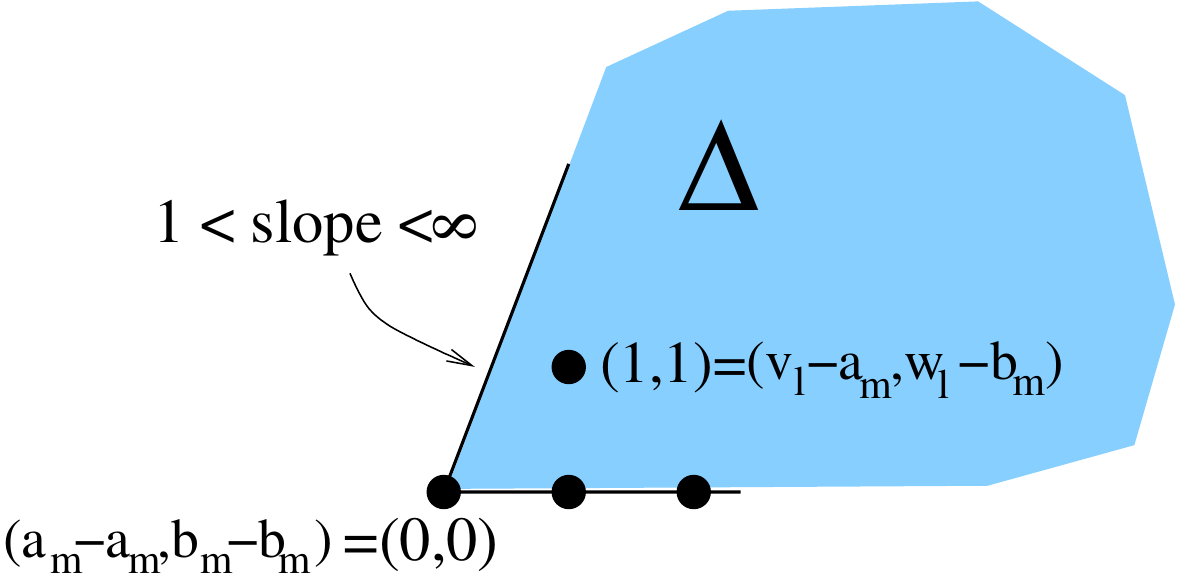}
\end{equation}
where the bottom edge corresponds to the toric divisor at whose intersection with $C_{\ul}$ we wish to compute $\mathrm{Tame}\left(\{-x,-y\}|_{C_{\ul}^*}\right)\in \CC^*$. Since the edge polynomial is $(1+x)^c$, this intersection occurs at $(-1,0)$, so the Tame symbol is $1$.
\end{proof}

Now set $\mathscr{R}_{\ul} := AJ^{2,2}(\Xi|_{C_{\ul}}) \in \mathit{Hom}(H_1(C_{\ul},\ZZ),\CC/\ZZ(2)).$  Picking any vertex $p_m$ of $\Delta$, we can (via unimodular coordinate change) put it in the position \eqref{fig!}.  In the new coordinates (still denoted $(x,y)$), $C_{\ul}$ is cut out by an equation of the form 
\begin{flalign*}
0=\tilde{\Phi}_{\ul}:&=x^{-a_m}y^{-b_m} \Phi_{\ul}(x,y)\\
&=(1+x)^{\kappa_m} + y\{ \Psi_m(x,y)-\Sigma_{j=1}^g \l_j x^{v_j-a_m} y^{w_j - b_m -1}\},
\end{flalign*} 
and acquires a node at $(0,0)$ as $\lambda_{\ell}\to \infty$.  (Note that $\ell$ is determined by $m$.)  The corresponding vanishing cycle $\alpha_m$ has image $|x|=|y|=\epsilon$ under $H_1(C_{\ul}),\ZZ)\overset{\text{Tube}}{\longrightarrow} H_2(\PP_{\Delta}\setminus C_{\ul})$ for large $|\l_{\ell}|$.

\begin{prop} \label{prop6.1}
For $\ay\l_{\ell}\in \mathfrak{H}$ and $|\l_{\ell}|\gg 0$, and $\l_{j\neq \ell}$ sufficiently small,\footnote{e.g. if $c:=|\Delta \cap \ZZ^2|$, then $|\l_{\ell}|> c\epsilon^{-2}$ and $|\l_{j\neq \ell}|<\frac{1}{c\epsilon^3}$ will do.} we have $$\mathscr{R}_{\ul}(\alpha_m) = 2\pi\ay ( -\log(\l_{\ell})+\textstyle{\sum}_{k\geq 1}\frac{1}{k} [\Psi_{\ul,\ell}^k]_{\underline{0}} ) \in \CC/\ZZ(2),$$ where $\Psi_{\ul,\ell}:=\frac{-1}{\l_{\ell}}\left( x^{-v_{\ell}}y^{-w_{\ell}}\Phi_{\ul} + \l_{\ell}\right)$ and $[\cdot]_{\underline{0}}$ takes the constant term in a Laurent polynomial.
\end{prop}

\begin{proof}
We use the notation $R\{f_1,f_2\}=\log(f_1)\frac{df_1}{f_1}-2\pi\ay\log(f_2)\delta_{T_{f_1}}$ and $R\{f_1,f_2,f_3\}=\log(f_1)\frac{df_2}{f_2}\wedge\frac{df_3}{f_3} + 2\pi\ay\log(f_2)\frac{df_3}{f_3}\delta_{T_{f_1}}+(2\pi\ay)^2\log(f_3)\delta_{T_{f_1}\cap T_{f_2}}$ for $R_2$ and $R_3$ with $f_i$ replacing $z_i$.  Writing $D$ for the bottom-edge divisor in \eqref{fig!}, we have 
$$\mathrm{Tame}_D\{\tilde{\Phi}_{\ul},-x,-y\}=\{\tilde{\Phi}_{\ul}(x,0),-x\}=\{(1+x)^c,-x\} (=1).$$
So writing $\Gamma=\{|x|=\epsilon\geq |y|\}$ ($\implies \alpha_m = \Gamma \cap C_{\ul}$) gives $\mathscr{R}_t(\alpha_m) =$
\begin{flalign*}
\int_{\alpha_m}R\{-x,-y\}&= \int_{\Gamma} R\{-x,-y\}\cdot \delta_{C_{\ul}} \\
&=\tfrac{-1}{2\pi\ay}\int_{\Gamma}d[R\{\tilde{\Phi}_{\ul},-x,-y\}] - \int_{\Gamma} R\{ (1+x)^c,-x\}\cdot \delta_D \\
&= \tfrac{-1}{2\pi\ay}\int_{\partial\Gamma} R\{\tilde{\Phi}_{\ul},-x,-y\}-\cancelto{0}{\int_{|x|=\epsilon} R\{ (1+x)^c,-x\}} \\
&=\tfrac{-1}{2\pi\ay}\int_{|x|=|y|=\epsilon} R\{ x^{-v_{\ell}}y^{-w_{\ell}}\Phi_{\ul},-x,-y\} \\
&=\tfrac{-1}{2\pi\ay} \int_{|x|=|y|=\epsilon} R\{\l (1-\Psi_{\ul,\ell}),-x,-y\} \\
&=\tfrac{-1}{2\pi\ay}\int_{|x|=|y|=\epsilon} \{\log(\l) + \log(1-\Psi_{\ul,\ell})\}\tfrac{dx}{x}\wedge\tfrac{dy}{y} \\
&=-2\pi\ay \log(\l) + 2\pi\ay \sum_{k\geq 1}\int_{|x|=|y|=\epsilon} \Psi_{\ul,\ell}^k \tfrac{dx}{x}\wedge\tfrac{dy}{y} 
\end{flalign*}
modulo $\ZZ(2)$. Here only the first term of $R_3$ enters since  $T_{\l(1-\Psi)}\cap |x|=|y|=\epsilon$ is empty under the given assumptions.
\end{proof}
Returning to the more specific scenario at the beginning of this section, if $g=1$ and $\l_1=:\l=:\tfrac{1}{t}$, then $\Phi_{\l}=\phi(x,y)-\l$ and $\Psi_{\l,1}=t\phi(x,y)$, so that (writing $\mathscr{R}_t$ instead of $\mathscr{R}_{\l}$), Prop. \ref{prop6.1} yields:

\begin{cor} \label{cor6}
If $\Delta$ is reflexive, then the $\alpha_m$ are all homologous $(=:\alpha)$, and $\mathscr{R}_t(\alpha)\underset{\ZZ(2)}{\equiv} 2\pi\ay \left( \log(t)+\sum_{k\geq 1}\frac{[\phi^k]_{\underline{0}}}{k}\right)$ for $t$ small in the right-half-plane.
\end{cor}

It remains to compute $\mathscr{R}_t(\beta)$ for a cycle $\beta$ complementary to $\alpha$ (so that $\ZZ\langle \alpha,\beta\rangle=H_1(E_t,\ZZ)$), which we shall do for the local $\PP^2$ setting only: $\Delta$ the convex hull of $\{ (1,0),(0,1),(-1,-1)\}$, and $\phi=x+y+x^{-1}y^{-1}$. Taking $t>0$ small, write $(0<)\,x_0(t)<x_-(t)<x_+(t)<\infty$ for the branch points of $$E_{t\,(=\l^{-1})} :\; y^2 + (x-\l)y + x^{-1}=0$$over $\PP^1_x$, and $y^{\pm}(x)=\tfrac{1}{2}\{(\l-x)\pm \sqrt{(x-\l)^2-4x^{-1}}\}$. Then $\beta$ [resp. $\alpha$] is given by the difference of paths (on the two branches) between $x_0(t)$ and $x_-(t)$ [resp. $x_-(t)$ and $x_+(t)$].

Now $T_{-x}=\RR_{>0} \subset \PP^1_x$, so taking the $y^+$- [resp. $y^-$-] branch of $\beta$ to run from $x_0$ to $x_-$ [resp. $x_-$ to $x_0$] in $\mathfrak{H}$ [resp. $-\mathfrak{H}$], we have $\beta\cap T_{-x}=(x_0,y_0)\cup (x_-,y_-)$; moreover, $\log(-x)=\log(x)\mp \ay\pi$ on the $y^{\pm}$-branch of $\beta$. The upshot is that
\begin{flalign*}
\int_{\beta} R\{-x,-y\}|_{E_t} &= \int_{\beta} \log(-x)\frac{dy}{y} - 2\pi\ay \sum_{\beta\cap T_{-x}}\log(y) \\
&= -\int_{x_0(t)}^{x_-(t)} \log(x)\mathrm{dlog}\left(\frac{y^+(x)}{y^-(x)}\right) \\
&=\int_{x_0(t)}^{x_-(t)}\log\left(\frac{y^+(x)}{y^-(x)}\right)\frac{dx}{x} \\
&=\int_{x_0(t)}^{x_-(t)} \log\left( \frac{1+\sqrt{1-\xi}}{1-\sqrt{1-\xi}}\right) \frac{dx}{x} 
\end{flalign*}
where $\xi=\frac{4t^2}{x(1-xt)^2}$.  Writing for $\xi \in (0,1)$ $$\log\left(\frac{1+\sqrt{1-\xi}}{1-\sqrt{1-\xi}}\right) + \log\left(\frac{\xi}{4}\right)=:-\sum_{m\geq 1}\alpha_m \xi^m,$$ the above integral decomposes into $$-2\log(t)\int_{x_0}^{x_-}\tfrac{dx}{x}+\int_{x_0}^{x_-}\log(x)\tfrac{dx}{x}+2\int_{x_0}^{x_-}\log(1-xt)\tfrac{dx}{x}-\sum_{m\geq 1}\alpha_m \int_{x_0}^{x_-}\xi^m\tfrac{dx}{x}.$$ Using the approximations $x_0 \simeq 4t^2(1+8t^3)$ and $x_-\simeq t^{-1}(1-2 t^{\frac{3}{2}}-2t^3)$, a lengthy direct computation gives that $$\mathscr{R}_t(\beta)=\tfrac{9}{2}\log^2(t)-\tfrac{\pi^2}{2}+\mathcal{O}(t\log(t)).$$

Let $\delta_t := t\tfrac{d}{dt}$. By a general result of \cite{DK1}, one knows that $\nabla_{\delta_t}\mathscr{R}_t=[\omega_t]$, where $$\omega_t:=\mathrm{Res}_{E_t}\left(\frac{\frac{dx}{x}\wedge \frac{dy}{y}}{1-t\phi(x,y)}\right)$$ has its periods $\omega_t(\gamma):=\int_{\gamma}\omega_t $ annihilated by the Picard-Fuchs operator$$\mathcal{L}=\delta_t^2 - 27t^3 (\delta_t +1)(\delta_t + 2).$$ The regulator periods $\mathscr{R}_t(\gamma)$ are therefore killed by $\mathcal{L}\circ\delta_t$. Since $\mathcal{L}(\cdot)=0$ is known to have basis of solutions
\begin{flalign*}
\pi_1&=\sum_{n\geq 0} a_n t^{3n}\\
\pi_2&=3\log(t)\pi_1 + \sum_{n\geq 1}a_n b_n t^{3n}
\end{flalign*}
with $a_n = \tfrac{(3n)!}{(n!)^3}$ and $b_n = \sum_{k=0}^{n-1} \left( \tfrac{3}{3k+1}+\tfrac{3}{3k+2}-\frac{2}{k+1}\right)$, it now follows that (writing $B_n = b_n -\tfrac{1}{n}$)
\begin{flalign*}
\mathscr{R}_t(\alpha)\underset{\ZZ(2)}{\equiv}& 2\pi\ay \left( \log(t)+\sum_{n\geq 1}\frac{a_n}{3n}t^{3n}\right)\\
\mathscr{R}_t(\beta)\underset{\ZZ(2)}{\equiv}& \frac{9}{2}\log^2(t) + 3\log(t)\sum_{n\geq 1}\frac{a_n}{n}t^{3n}+\sum_{n\geq 1}\frac{a_n B_n}{n}t^{3n}-\frac{\pi^2}{2}\\
\omega_t(\alpha)=\,& 2\pi\ay \sum_{n\geq 0}a_n t^{3n}\\
\omega_t(\beta)=\,& 9\log(t)\sum_{n\geq 0}a_n t^{3n} +3 \sum_{n\geq 1}a_n b_n t^{3n}
\end{flalign*}
for $0<|t|<\frac{1}{3}$. For the truncated normal function, this yields (modulo $\ZZ(2)\otimes \{\omega_t\text{-periods}\}$)
\begin{flalign*}
\nu(t)=&\langle\tfrac{\omega_t}{2\pi\ay},\mathscr{R}_t\rangle = \tfrac{1}{2\pi\ay}\left(\mathscr{R}_t(\alpha)\omega_t(\beta)-\mathscr{R}_t(\beta)\omega_t(\alpha)\right)\\
=&\tfrac{9}{2}\log^2(t)(1+6t^3)+3\log(t)(9t^3)+\tfrac{\pi^2}{2}+(3\pi^2-9)t^3+\mathcal{O}(t^6\log^2 t).
\end{flalign*}

\begin{rem}
This is closely related to computations in \cite{Ho} and \cite{MOY}; the main difference -- and the salient result here -- is the identification of $\frac{\pi^2}{2}$ as the correct torsion offset for our motivically defined $\nu$.
\end{rem}

\curraddr{${}$\\
\noun{Department of Mathematics, Campus Box 1146}\\
\noun{Washington University in St. Louis}\\
\noun{St. Louis, MO} \noun{63130, USA}}

\email{${}$\\
\emph{e-mail}: matkerr@math.wustl.edu}

\curraddr{\noun{${}$}\\
\noun{School of Mathematical Sciences}\\
\noun{University of Science and Technology of China}\\
\noun{Hefei, 230026, CHINA}}

\email{\emph{${}$}\\
\emph{e-mail}: limuxi@ustc.edu.cn}
\end{document}